\documentclass{amsart} 
\usepackage{latexsym}
\usepackage{amsfonts}
\usepackage{amssymb}
\usepackage{amsmath}
\usepackage{amscd}
\usepackage{graphicx}
\usepackage{eucal}
\usepackage[all]{xy}
\newcommand{\bs}{\boldsymbol}
\newcommand{\mb}{\mathbf}
\renewcommand{\dot}{\centerdot}

\begin{document}
\newtheorem{thm}{Theorem}[section]
\newtheorem{proposition}[thm]{Proposition}
\newtheorem{lemma}[thm]{Lemma}
\newtheorem{corollary}[thm]{Corollary}

\newtheoremstyle{example}{\topsep}{\topsep}%
     {}
     {}
     {\bfseries}
     {.}
     {2pt}
     {\thmname{#1}\thmnumber{ #2}\thmnote{ #3}}

   \theoremstyle{example}
   \newtheorem{notation}[thm]{Notation}
   \newtheorem{definition}[thm]{Definition}
   \newtheorem{remark}[thm]{Remark}

\title{The monoidal structure of strictification}
\author{Nick Gurski}
\address{School of Mathematics and Statistics, University of Sheffield, Sheffield, UK, S3 7RH}
\email{nick.gurski@sheffield.ac.uk}

\maketitle

\begin{abstract}
We study the monoidal structure of the standard strictification functor $\textrm{st}:\mathbf{Bicat} \rightarrow \mathbf{2Cat}$.  In doing so, we construct monoidal structures on the 2-category whose objects are bicategories and on the 2-category whose objects are 2-categories.
\end{abstract}


\section*{Introduction}

The study of coherence in higher category theory is the study of the relationship between strict structures and their weak variants.  A strict structure is one which imposes many axioms, while a weak structure is one in which axioms are replaced by coherent isomorphisms whenever possible.  The classic example of a coherence theorem is Mac Lane's result \cite{mac} that every monoidal category is equivalent to a strict monoidal category.  Here, the notion of a strict monoidal category includes the requirement that $(x \otimes y) \otimes z = x \otimes (y \otimes z)$, while for a general monoidal category we only insist on there being a natural isomorphism between these two composite tensors (satisfying certain other axioms).  Thus Mac Lane's theorem states that when considering monoidal structures, there is no essential difference between the strict sort and the more general sort having isomorphisms in place of the usual monoid axioms.

Moving more firmly into the realm of higher category theory, the coherence theorem for monoidal categories can be modified to become a coherence theorem for bicategories.  This theorem (see \cite{mp}) states that every bicategory is biequivalent to a strict 2-category.  Once again, this particular coherence theorem shows how every ``weakly defined'' structure of a particular sort is essentially the same as a ``strictly defined'' version of the same sort of structure.  Additionally, these theorems can be extended to functors as well, proving that every weak functor between bicategories can be replaced, up to biequivalence, by a strict 2-functor between 2-categories.

In fact, this kind of pattern has been studied intensely by 2-category theorists.  One begins with a 2-monad $T$ on a 2-category $K$; this is the appropriate notion of a kind of algebraic structure (given by $T$) on objects of the 2-category $K$.  Then there are different notions of algebra for $T$, using the 2-dimensional nature of $K$ in different ways.  Thus one can study lax $T$-algebras, pseudo-$T$-algebras, and strict $T$-algebras, as well as various kinds of algebra maps between them.  Within this context, there are two kinds of coherence theorems.  First, one might show that the inclusion of stricter algebras or maps into weaker ones has a left adjoint.  Such a theorem shows that there is a classifier for weak maps, and can be interpretted as a kind of coherence theorem for morphisms.  Second, one might show that the adjunction has a unit with components which are equivalences, giving a method for turning weak algebras into strict ones up to equivalence.  This abstract setting captures the most important aspects of the example of monoidal categories.  The reader interested in pursuing this further should read \cite{bkp, la2}.

Moving up to the study of various kinds of 3-categories, the situation becomes much more interesting.  The coherence theorem for tricategories states that every tricategory is triequivalent to a \textbf{Gray}-category, but that there are tricategories which are not triequivalent to strict 3-categories.  Strict 3-categories are exactly what the name implies, a fully strict version of a three-dimensional category with associativity and unit axioms imposed directly at each dimension.  A tricategory is, in contrast, a fully weak version of a 3-category, introduced first by Gordon, Power, and Street in \cite{gps} and then modified slightly by the author in \cite{gthesis}.  A \textbf{Gray}-category is an intermediate notion, and is sometimes referred to as a semi-strict notion of 3-category because while it retains many of the axioms that hold in a 3-category it does have one notable exception:  the interchange law for 2-cells only holds up to a coherence isomorphism, which then itself satisfies some axioms reminiscent of the braid relations.  \textbf{Gray}-categories can be defined using enriched category theory through the introduction of the Gray tensor product, and it is the relationship between strictification of bicategories and the Gray tensor product that is the focus of this paper.

The particular interest in showing that strictification is monoidal arises from the proof of coherence for tricategories.  One version of coherence for bicategories can be proved in a single step by constructing the Yoneda embedding $B \hookrightarrow [B^{op}, \mb{Cat}]$.  By general results, the target bicategory for this functor is actually a 2-category since $\mb{Cat}$ is a 2-category, so the Yoneda lemma for bicategories embeds a bicategory within a strict 2-category; taking the essential image shows that every bicategory is biequivalent to a strict 2-category.  The same strategy will not work without modification in dimension three.  Here, we would get a Yoneda embedding
$T \hookrightarrow [T^{op}, \mb{Bicat}]$, but since $\mb{Bicat}$ is not very strict as a tricategory, the target of this embedding will also not be as strict as desired.  Instead, one replaces a general tricategory $T$ with what is called a cubical tricategory $\overline{T}$, and then it is possible to embed any cubical tricategory in a \textbf{Gray}-category.  The original construction of this embedding in \cite{gps} uses the \textbf{Gray}-category of pre-representations, while the author's proof in \cite{gthesis, gbook} once again uses a Yoneda lemma for cubical tricategories.

For the purposes of this paper, the crucial step in this proof is the construction of a cubical tricategory $\overline{T}$ which is triequivalent to the original tricategory $T$.  Now a cubical tricategory is one in which the hom-bicategories are actually 2-categories, and in which the functors giving composition and units are cubical.  In order to strictify the hom-bicategories, one applies the strictification functor $\textrm{st}:\mb{Bicat} \rightarrow \mb{2Cat}$.  This strictification functor, constructed using the coherence theorems for bicategories and functors, has the property that it comes with a comparison map
\[
\textrm{st}X \times \textrm{st}Y \rightarrow \textrm{st}(X \times Y).
\]
The insight in \cite{gps} is that this comparison functor is cubical, and so the cubical tricategory $\overline{T}$ can be constructed by applying st to the hom-bicategories of $T$ and then using this comparison functor to define composition.

Studying the monoidal structure on this functor was the impetus for this paper.  There are errors in both \cite{gps} and \cite{gthesis} claiming that the comparison functor given above equips the functor st with a lax monoidal structure.  This is in fact false, as the maps above are not natural in the usual sense.  They are natural up to an invertible icon, which is a kind of 2-cell that shows up in various constructions involving bicategories and 2-categories (see \cite{la3, lp}).  Thus the goal of this paper is to construct relevant monoidal structures on the 2-categories of bicategories, functors, and icons on the one hand and 2-categories, 2-functors, and icons on the other hand such that strictification has a monoidal structure.  This is achieved in Theorem \ref{stmonoidal}.

One of the major goals of ``low-dimensional'' category theory is a good monoidal structure on the category of \textbf{Gray}-categories, as an example see \cite{crans} for an attempt to construct such a monoidal structure directly.  The results in this paper suggest that an alternate approach to this problem consists of studying the monoidal structure on the strictification process for tricategories.  This will necessarily have to take place in a higher dimensional context as strictification for tricategories will not be a monoidal functor between monoidal \textit{categories} just as the functor st is not, but it will also likely not be a monoidal functor between monoidal \textit{tetracategories}, once again just as st is monoidal on the level of 2-categories.  The results in this paper suggest that one avenue of approach would be to start with the low-dimensional structures constructed in \cite{gg}, and then to study the monoidal structure on the strictification $T \mapsto \textrm{Gr}T$ of \cite{gthesis, gbook}.

This paper will proceed as follows.  The first section contains a review of the Gray tensor product of 2-categories, but omits most of the proofs.  The interested reader should consult the original work of Gray \cite{gray}, the monograph of Gordon, Power, Street \cite{gps}, or the author's forthcoming book \cite{gbook}.  The second section reviews some basic notions from the theory of monoidal bicategories, as well as giving some terminology for a stricter variant (strongly monoidal 2-categories) that will appear naturally.  The third section begins with the definition of an icon, and then goes on to prove that icons appear naturally as the 2-cells in both a strongly monoidal 2-category of bicategories and a strongly monoidal 2-category of 2-categories.  The fourth section studies the structure of strictification as a 2-functor, and the final section is devoted to providing this functor with a monoidal structure.  This can be done in a straightforward, calculational fashion, but we have chosen to prove that strictification is a monoidal functor on the level of 2-categories using the theory of doctrinal adjunction \cite{k}.

The author would like to thank Steve Lack, John Power, and Ross Street for conversations and correspondence which have contributed to the research in this paper, as well as the referees for comments which have helped shape the final form of this paper.

\section{The Gray tensor product}

The Gray tensor product  gives a monoidal structure on the category of 2-categories and 2-functors which differs from the usual Cartesian structure.  Here we give a quick overview of the generators-and-relations definition of the Gray tensor product of a pair of 2-categories, $A \otimes B$; we do not give details, and refer the reader to \cite{gps, gray, gbook} for more information.

\begin{notation}  We will denote horizontal composition of any kind in a 2-category by $*$ when such a symbol is necessary, and vertical composition of 2-cells by concatenation.
\end{notation}

Let $A,B$ be 2-categories.  Then $A \otimes B$ is the 2-category with
\begin{itemize}
\item objects $a \otimes b$ where $a \in A, b \in B$;
\item 1-cells generated by
\[
\begin{array}{rl}
f \otimes 1: a \otimes b \rightarrow a' \otimes b, & f:a \rightarrow a' \textrm{ in } A, b \in B, \\
1 \otimes g: a \otimes b \rightarrow a \otimes b', & a \in A, g:b \rightarrow b' \textrm{ in }B,
\end{array}
\]
subject to the relations
\[
\begin{array}{rcl}
(f \otimes 1)*(f' \otimes 1) & = & (f*f') \otimes 1, \\
(1 \otimes g)*(1 \otimes g') & = & 1 \otimes (g*g'),
\end{array}
\]
where the composite on the lefthand side is taken in $A \otimes B$ while the composite on the righthand side is in $A$ or $B$, respectively;
\item 2-cells generated by
\[
\begin{array}{c}
\alpha \otimes 1:f \otimes 1 \Rightarrow f' \otimes 1, \\
1 \otimes \beta:1 \otimes g \Rightarrow 1 \otimes g', \\
\Sigma_{f,g}:(f \otimes 1)*(1 \otimes g) \Rightarrow (1 \otimes g)*(f \otimes 1),
\end{array}
\]
where $\alpha$ is a 2-cell in $A$ and $\beta$ is a 2-cell in $B$.
\end{itemize}
These 2-cells are subject to further relations.  Some of these relations require that the function $A \rightarrow A \otimes B$ sending $\delta$ to $\delta \otimes b$ is a 2-functor for every $b \in B$, and analogously if the $A$-variable is fixed.  There are also naturality axioms for the 2-cells $\Sigma_{f,g}$, together with some braid-like axioms governing composites of different $\Sigma$'s.

The Gray tensor product serves a variety of functions, one of those functions being to classify cubical functors as defined below.

\begin{definition}
A functor $F: A_{1} \times A_{2} \times \cdots A_{n} \rightarrow B$ is
\textit{cubical} if the following condition holds:\\
if $(f_{1}, f_{2}, \ldots, f_{n})*(g_{1}, g_{2}, \ldots, g_{n})$ is a
composable
pair of morphisms in the 2-category $A_{1} \times A_{2} \times \cdots A_{n}$ such that for
all
$i > j$, either $g_{i}$ or $f_{j}$ is an identity map, then the comparison
2-cell
\[
\phi: F(f_{1}, f_{2}, \ldots, f_{n})*F(g_{1}, g_{2}, \ldots,
g_{n}) \Rightarrow F \Big( (f_{1}, f_{2}, \ldots, f_{n})*(g_{1}, g_{2}, \ldots, g_{n})
\Big)
\]
is an identity.
\end{definition}

\begin{remark}
This definition is equivalent to the condition that, for any $1 \leq j \leq n$, the comparison 2-cell
\[
 (1, 1, \ldots, 1, f_{j}, f_{j+1}, \ldots, f_{n})*F(g_{1}, \ldots, g_{j}, 1, \ldots, 1) \Rightarrow F(g_{1}, \ldots, g_{j-1}, f_{j}g_{j}, f_{j+1}, \ldots, f_{n})
\]
is the identity 2-cell.
\end{remark}

\begin{remark}
1.  Every cubical functor strictly preserves identity 1-cells. \\
2.  A cubical functor of one variable is necessarily a 2-functor.
\end{remark}

The relationship between the Gray tensor product and cubical functors is described by the theorem below.  Before stating it, we require a definition.

\begin{definition}
Let $A_{1}, \ldots, A_{n}, B$ be $2$-categories.  Then $\mathbf{2Cat}_{c}(A_{1}, \ldots, A_{n}; B)$ is the set of cubical functors $A_{1} \times \cdots \times A_{n} \rightarrow B$.
\end{definition}

\begin{thm}\label{cubicalmulticat}
Let $A$, $B$, and $C$ be 2-categories.  There is a cubical functor
\[
c:A \times B \rightarrow A \otimes B,
\]
natural in $A$ and $B$, such that composition with $c$ induces an isomorphism
\[
\mathbf{2Cat}_{c}(A, B; C) \cong \mathbf{2Cat}(A \otimes B, C).
\]
\end{thm}
\begin{proof}[Sketch]
The functor $c$ is defined by sending $(a,b)$ to $a \otimes b$, $(f,1)$ to $f \otimes 1$, $(1,g)$ to $1 \otimes g$, and similarly for 2-cells.  We force this functor to be cubical, and then it is easy to check the isomorphism in the theorem.
\end{proof}

We will need the following lemma later.

\begin{lemma}\label{universalboobieq}
The universal cubical functor $c:A \times B \rightarrow A \otimes B$ is a bijective-on-objects biequivalence.
\end{lemma}
\begin{proof}[Sketch]
First, the explicit construction of $c$ makes it obvious that it is bijective-on-objects.  Since every strict functor is cubical, we use the universal property of $c$ to produce a unique 2-functor $\pi$ making the diagram below commute.
\[
\xy
{\ar^{c} (0,0)*+{A \times B}; (50,0)*+{A \otimes B} };
{\ar^{\pi} (50,0)*+{A \otimes B}; (50,-20)*+{A \times B} };
{\ar_{1} (0,0)*+{A \times B}; (50,-20)*+{A \times B} };
\endxy
\]
All that remains is to show that $c \pi$ is equivalent to $1_{A \otimes B}$, a calculation we leave for the reader.
\end{proof}

\section{Monoidal bicategories}

We begin by reminding the reader of the definitions of monoidal bicategory and monoidal functor, and then go on to state a stricter, \textbf{Cat}-enriched version of these definitions that will play a role later.  Because the monoidal structures we will construct are actually monoidal \textit{category} structures extended to take into account the 2-cells, these stricter kinds of monoidal bicategories will arise naturally.

\begin{definition}
A \textit{monoidal bicategory} is a one-object tricategory in the fully algebraic sense of \cite{gthesis}.  Thus a monoidal bicategory consists of
\begin{itemize}
\item a bicategory  $B$;
\item a functor $\otimes: B \times B \rightarrow B$;
\item a functor $I: 1 \rightarrow B$;
\item adjoint equivalence $\bs{a}, \bs{l}$, and $\bs{r}$ as in the definition of a tricategory; and
\item invertible modifications $\pi, \mu, \lambda$, and $\rho$ as in the definition of a tricategory
\end{itemize}
all subject to the tricategory axioms.
\end{definition}

\begin{definition}
A \textit{monoidal functor} between monoidal bicategories is a functor between the one-object tricategories.  Such a monoidal functor consists of
\begin{itemize}
\item a functor of the underlying bicategories $F:B_{1} \rightarrow B_{2}$;
\item adjoint equivalences $\bs{\chi}$ and $\bs{\iota}$ as in the definition of weak functor between tricategories; and
\item invertible modifications $\omega, \delta$, and $\gamma$ as in the definition of weak functor
\end{itemize}
all subject to axioms which are identical to the tricategory functor axioms aside from source and target considerations.
\end{definition}

\begin{definition}
A \textit{lax monoidal functor} between monoidal bicategories consists of
\begin{itemize}
\item a functor of the underlying bicategories $F:B_{1} \rightarrow B_{2}$;
\item transformations $\chi$ and $\iota$ as in the definition of weak functor between tricategories; and
\item invertible modifications $\omega, \delta$, and $\gamma$ as in the definition of weak functor
\end{itemize}
all subject to axioms which are identical to the tricategory functor axioms aside from source and target considerations.
\end{definition}

\begin{remark}
Note that the difference in the two definitions above lies with $\chi, \iota$.  In a monoidal functor, these are required to be part of adjoint equivalences $\chi \dashv_{eq} \chi^{\dot}, \iota \dashv_{eq} \iota^{\dot}$, while in a lax monoidal functor they are not.  In particular, this implies that, for a lax monoidal functor, the components on objects
\[
\begin{array}{rrcl}
\chi: & Fx \otimes_{2} Fy &  \rightarrow & F(x \otimes_{1} y), \\
\iota: & I_{Fx} & \rightarrow & FI_{x}
\end{array}
\]
are not necessarily equivalences in $B_{2}$.  The lax monoidal functors we use here are the same as the lax homomorphisms of \cite{gg}.
\end{remark}

\begin{remark}
For completeness, we should also define symmetric monoidal bicategories and the appropriate notion of monoidal functor between those, as the main theorem of this paper will be that a certain functor is a symmetric monoidal one.  We omit these details and definitions here, as it is actually the monoidal aspects of this functor that are of primary interest, not the relationship with the symmetry.  For readers interested in precise definitions, we refer to \cite{ds, gur2, sp}.
\end{remark}

\begin{definition}
A \textit{strongly monoidal 2-category} consists of
\begin{itemize}
\item a 2-category $A$,
\item a 2-functor $m:A \times A \rightarrow A$,
\item an object $I$ of $A$,
\item a 2-natural isomorphism $\mu$,
\[
\xy
{\ar^{m \times 1} (0,0)*+{A \times A \times A}; (40,0)*+{A \times A} };
{\ar^{m} (40,0)*+{A \times A}; (40,-15)*+{A} };
{\ar_{1 \times m} (0,0)*+{A \times A \times A}; (0,-15)*+{A \times A} };
{\ar_{m} (0,-15)*+{A \times A}; (40,-15)*+{A} };
(20, -7)*{\Downarrow \mu}
\endxy
\]
\item 2-natural isomorphisms $\lambda, \rho$,
\[
\xy
{\ar^{1 \times I} (0,0)*+{A \times *}; (35,0)*+{A \times A} };
{\ar^{m} (35,0)*+{A \times A}; (35,-20)*+{A} };
{\ar_{\cong} (0,0)*+{A \times *}; (35,-20)*+{A} };
(22,-7)*{\Downarrow \lambda};
{\ar^{I \times 1} (55,0)*+{* \times A}; (90,0)*+{A \times A} };
{\ar^{m} (90,0)*+{A \times A}; (90,-20)*+{A} };
{\ar_{\cong} (55,0)*+{* \times A}; (90,-20)*+{A} };
(77,-7)*{\Downarrow \rho}
\endxy
\]
\end{itemize}
such that the usual monoidal category axioms hold.
\end{definition}

\begin{remark}
This definition is merely the $\mathbf{Cat}$-enriched version of the definition of a monoidal category.
\end{remark}

\begin{proposition}\label{underlyingbicat}
Every strongly monoidal 2-category has an underlying monoidal bicategory.
\end{proposition}
\begin{proof}
The underlying bicategory, tensor product, and unit are all the same as those in the strongly monoidal 2-category.  The associativity adjoint equivalence is given by  $\mu$ and $\mu^{-1}$ with unit and counit both being the identity; the analogous statements hold for the left and right unit adjoint equivalences.  The remaining data ($\pi, \mu, \lambda, \rho$ in the notation of monoidal bicategories, not to be confused with the data above) can all be taken to be identities as they correspond to axioms or consequences of the axioms for strongly monoidal 2-categories.
\end{proof}

\begin{definition}\label{products}
Let $A$ be a 2-category.
\begin{enumerate}
\item A \textit{product} $a \times b$ of objects $a,b \in A$ is an object $a \times b \in A$ together with 2-natural isomorphisms
\[
A(x, a \times b) \cong A(x,a) \times A(x,b)
\]
for all objects $x \in A$.
\item A \textit{terminal object} of $A$ is an object $* \in A$ together with 2-natural isomorphisms
\[
A(x,*) \cong 1
\]
for all objects $x \in A$, where $1$ denotes a terminal category.
\end{enumerate}
\end{definition}

\begin{proposition}\label{2catprods}
Let $A$ be a 2-category with all binary products and a terminal object.  Then a choice of product for every pair of objects and a choice of terminal object equip $A$ with the structure of a strongly monoidal 2-category.
\end{proposition}
\begin{proof}
The proof here is identical to the proof that categories with chosen binary products and a chosen terminal object are monoidal.
\end{proof}

\begin{definition}
A \textit{\textbf{Cat}-lax monoidal functor} $F:A_{1} \rightarrow A_{2}$ between strongly monoidal 2-categories is a 2-functor $F$ equipped with
\begin{itemize}
\item 1-cells
\[
\chi_{x,y}:Fx \otimes_{2} Fy \rightarrow F(x \otimes_{1} y)
\]
in $A_{2}$ for every pair of objects $x,y$ in $A_{1}$, which are 2-natural in $x,y$; and
\item a 1-cell
\[
\iota:I_{2} \rightarrow FI_{1}
\]
in $A_{2}$ where $I_{j}$ denotes the unit object in $A_{j}$.
\end{itemize}
These are required to satisfy the usual axioms for a lax monoidal functor between monoidal categories.
\end{definition}

\begin{definition}
A \textit{\textbf{Cat}-oplax monoidal functor} $F:A \rightarrow B$ between strongly monoidal 2-categories is a 2-functor $F$ equipped with
\begin{itemize}
\item 1-cells
\[
\chi_{x,y}:F(x \otimes_{1} y) \rightarrow Fx \otimes_{2} Fy
\]
in $A_{2}$ for every pair of objects $x,y$ in $A_{1}$, which are 2-natural in $x,y$; and
\item a 1-cell
\[
\iota:FI_{1} \rightarrow I_{2}
\]
in $A_{2}$ where $I_{j}$ denotes the unit object in $A_{j}$.
\end{itemize}
These are required to satisfy the usual axioms for an oplax monoidal functor between monoidal categories.
\end{definition}

\begin{remark}
Once again, these are merely the \textbf{Cat}-enriched versions of the definitions of lax monoidal functor and oplax monoidal functor.  In particular, we can consider the free monoidal $\mathcal{V}$-category 2-monad on the 2-category of $\mathcal{V}$-categories, $\mathcal{V}$-functors, and $\mathcal{V}$-natural transformations for suitably nice $\mathcal{V}$.  When $\mathcal{V} = \mb{Cat}$, the algebras for this 2-monad are the strongly monoidal 2-categories, and the (op)lax algebra morphisms are the \textbf{Cat}-(op)lax monoidal functors.  This observation will be useful when applying the theory of doctrinal adjunction.
\end{remark}

\begin{proposition}
Let $F:A_{1} \rightarrow A_{2}$ be a \textbf{Cat}-lax monoidal functor between strongly monoidal 2-categories.  Then $F$ can be equipped with the structure of a lax monoidal functor between the underlying monoidal bicategories.  Furthermore, if the components
\[
\begin{array}{rrcl}
\chi_{x,y}: & Fx \otimes_{2} Fy  & \rightarrow & F(x \otimes_{1} y) \\
\iota: & I_{2} & \rightarrow & FI_{1}
\end{array}
\]
are all internal equivalences in $A_{2}$, then this lax monoidal functor is a monoidal functor.
\end{proposition}
\begin{proof}
The underlying monoidal functor between monoidal bicategories is just $F$, and the transformations $\chi, \iota$ are those given in the \textbf{Cat}-lax structure.  As in the proof of Proposition \ref{underlyingbicat}, the top-dimensional data (in the notation of monoidal functors between monoidal bicategories, $\omega, \delta, \gamma$) can all be taken to be identities.  Since the only difference between a lax monoidal functor (in the sense used here) and a monoidal functor is the structure of $\chi, \iota$, the final claim follows from the observation that a pseudonatural transformation is part of an adjoint equivalence if and only if all the components are equivalences.
\end{proof}

We give two results here related to the process of turning an $\mb{Cat}$-oplax monoidal functor into a monoidal functor between monoidal bicategories.  The first is related to doctrinal adjunction \cite{k}, and while we do not use the result directly, we will need the formula given in the proof below.  The second will be used in the proof of the main result, Theorem \ref{stmonoidal}.

\begin{lemma}\label{oplaxtomonoidal1}
Let $(A_{1}, \otimes_{1}), (A_{2}, \otimes_{2})$ be strongly monoidal 2-categories, and let $F \dashv U$ be a 2-adjunction between them with $F:A_{1} \rightarrow A_{2}$.  Assume that the components of the unit and counit for this adjunction are all internal equivalences, and that $U$ has been given a $\mb{Cat}$-lax monoidal structure.  Then if the components
\[
\varphi_{xy}:Ux \otimes_{1} Uy \rightarrow U(x \otimes_{2} y)
\]
are all internal equivalences, so are the components
\[
\psi_{xy}:F(x \otimes_{1} y) \rightarrow Fx \otimes_{2} Fy
\]
for the corresponding $\mb{Cat}$-oplax monoidal structure for $F$.  If the 1-cell
\[
\varphi: I_{2} \rightarrow UI_{1}
\]
for the monoidal structure on $U$ is an internal equivalence in $A_{2}$, then so is the 1-cell
\[
\psi:FI_{2} \rightarrow I_{1}
\]
for the oplax structure on $F$.
\end{lemma}
\begin{proof}
The component $\psi_{xy}:F(x \otimes_{1} y) \rightarrow Fx \otimes_{2} Fy$ is given by the composite below:
\[
F(x \otimes_{1} y) \stackrel{F(\eta \otimes_{1} \eta)}{\longrightarrow} F(UFx \otimes_{1} UFy) \stackrel{F\varphi_{Fx,Fy}}{\longrightarrow} FU(Fx \otimes_{2} Fy) \stackrel{\varepsilon}{\longrightarrow} Fx \otimes_{2} Fy.
\]
If $\eta$ is an equivalence, then so is $\eta \otimes_{1} \eta$.  Since any 2-functor sends equivalences to equivalences, we get that each of the three 1-cells above is an equivalence, so their composite is as well.  The proof for the unit is similar.
\end{proof}

\begin{lemma}\label{oplaxtomonoidal2}
Let $(A_{1}, \otimes_{1}), (A_{2}, \otimes_{2})$ be strongly monoidal 2-categories, and let $F:A_{1} \rightarrow A_{2}$ be an $\mb{Cat}$-oplax monoidal functor between them.  If the components
\[
\begin{array}{c}
\varphi:FI_{1} \rightarrow I_{2}, \\
\varphi_{xy}:F(x \otimes_{1} y) \rightarrow Fx \otimes_{2} Fy
\end{array}
\]
are all internal equivalences in $A_{2}$, then their pseudo-inverses $\varphi_{0}^{\dot}, \varphi_{xy}^{\dot}$ are part of the structure of a monoidal functor $F:A_{1} \rightarrow A_{2}$ on the underlying monoidal bicategories.
\end{lemma}
\begin{proof}
Equip the pseudo-inverse pairs $(\varphi_{0}, \varphi_{0}^{\dot})$ and $(\varphi_{xy}, \varphi_{xy}^{\dot})$ with the structure of adjoint equivalences $\varphi \dashv_{eq} \varphi^{\dot}$.  Using this structure, the 2-naturality of the $\varphi_{xy}$ gives the $\varphi_{xy}^{\dot}$ the structure of a pseudonatural transformation $\otimes_{2} \circ (F \times F) \Rightarrow F \circ \otimes_{1}$.  The invertible 2-cells $\omega, \delta, \gamma$ for the monoidal functor $F$ are then the mates of the identity 2-cells for the three $\mb{Cat}$-oplax monoidal functor axioms.  We leave it to the reader to verify that the monoidal functor axioms are then a consequence of the triangle identities.
\end{proof}

\section{Icons}

While for many purposes the default notion of 2-cell between functors of bicategories is that of pseudonatural transformation (or some variant), there is another equally natural notion of 2-cell called an \textit{icon}.  Icons are essentially the many-object version of monoidal transformations, and thus appear when studying degeneracy and monoidal structures for higher categories (see \cite{cg1, cg2}).  Icons also appear naturally when viewing bicategories as the algebras for a 2-monad \cite{la3, lp}.  Here the underlying 2-category is that of category-enriched graphs, and the algebra 2-cells that the relevant 2-monad produces are precisely icons.  This section will give a brief review of the basic definitions required, and then use icons as the 2-cells in the construction of two strongly monoidal 2-categories.

\begin{definition}
Let $F,G:B \rightarrow C$ be lax functors between bicategories with constraints $\varphi_{0}, \varphi_{2}$ for $F$, and $\psi_{0}, \psi_{2}$ for $G$.  Assume that $F$ and $G$ agree on objects.  An \textit{icon} $\alpha:F \Rightarrow G$ consists of natural transformations
\[
\alpha_{ab}:F_{ab} \Rightarrow G_{ab}:B(a,b) \rightarrow C(Fa,Fb)
\]
(note here that we require $Fa=Ga,Fb=Gb$ so that the functors $F_{ab}, G_{ab}$ have a common target) such that the following diagrams commute.  (Note that we suppress the 0-cell source and target subscripts for the transformations $\alpha_{ab}$ and instead only list the 1-cell for which a given 2-cell is the component.)
\[
\xy
{\ar@{=>}^{\varphi_{0}} (0,0)*+{I_{Fa}}; (25,0)*+{FI_{a}} };
{\ar@{=>}^{\alpha_{I}} (25,0)*+{FI_{a}}; (25,-15)*+{GI_{a}} };
{\ar@{=>}_{\psi_{0}} (0,0)*+{I_{Fa}}; (25,-15)*+{GI_{a}} };
{\ar@{=>}^{\varphi_{2}} (50,0)*+{Ff*Fg}; (80,0)*+{F(f*g)} };
{\ar@{=>}^{\alpha_{f*g}} (80,0)*+{F(f*g)}; (80,-15)*+{G(f*g)} };
{\ar@{=>}_{\alpha_{f}*\alpha_{g}} (50,0)*+{Ff*Fg}; (50,-15)*+{Gf*Gg} };
{\ar@{=>}_{\psi_{2}} (50,-15)*+{Gf*Gg}; (80,-15)*+{G(f*g)} };
\endxy
\]
\end{definition}

\begin{definition}\label{deficon2cat}
The 2-category $\mathbf{Icon}$ is defined to have objects bicategories, 1-cells functors, and 2-cells icons between them.
\end{definition}

\begin{remark}
Note that this 2-category was called $\mathbf{Hom}$ in \cite{lp}.  To actually prove these cells, with their obvious composition laws, give a 2-category is a relatively basic calculation that we leave to the reader.
\end{remark}

The following lemma is an immediate consequence of the fact that the universal cubical functor $c:A \times B \rightarrow A \otimes B$ is a bijective-on-objects biequivalence, i.e., an internal equivalence in $\mathbf{Icon}$.

\begin{lemma}\label{cubicalicons}
Let $A,B,C$ be 2-categories, and let $F,G:A \times B \rightarrow C$ be a pair of cubical functors that agree on objects.  Then icons $\alpha:F \Rightarrow G$ are in natural bijection with icons $\tilde{\alpha}:\tilde{F} \Rightarrow \tilde{G}$ between the 2-functors induced by the universal property.
\end{lemma}

We now have the following easy corollary of Proposition \ref{2catprods}.

\begin{corollary}\label{Iconsmon}
The 2-category $\mathbf{Icon}$ of bicategories, functors, and icons can be given the structure of a strongly monoidal 2-category under products.
\end{corollary}
\begin{proof}
The usual product of bicategories and terminal bicategory are products and terminal in the sense of Definition \ref{products}.
\end{proof}

\begin{definition}\label{defgrayicon}
The 2-category $\mathbf{Gray}_{icon}$ is defined to be the locally full sub-2-category of $\mathbf{Icon}$ consisting of 2-categories, 2-functors, and icons.
\end{definition}

\begin{proposition}\label{Graysmon}
The Gray tensor product equips $\mathbf{Gray}_{icon}$ with the structure of a strongly monoidal 2-category.
\end{proposition}
\begin{proof}
The multiplication in $\mathbf{Gray}_{icon}$ is to be the Gray tensor product, so we must equip it with the structure of a 2-functor
\[
\otimes: \mathbf{Gray}_{icon} \times \mathbf{Gray}_{icon} \rightarrow \mathbf{Gray}_{icon}.
\]
It is clear that $\otimes(A,B) = A \otimes B$ and that $\otimes(F,G) = F \otimes G$, but we must define it on icons as well.  Let $\alpha:F \Rightarrow F', \beta:G \Rightarrow G'$ be icons.  Then $(F \otimes G)(a,b) = (Fa,Gb) = (F'a,G'b) = (F' \otimes G')(a,b)$ so $F \otimes G$ and $F' \otimes G'$ agree on objects, hence we may define an icon between them.  Since the 1-cells of $A \otimes B$ are generated by 1-cells of the form $f \otimes 1, 1 \otimes g$, we will give the components of the icon $\alpha \otimes \beta$ at these 1-cells and then extend over composition.  These components are defined as
\[
\begin{array}{rcl}
(\alpha \otimes \beta)_{f \otimes 1} & = & \alpha_{f} \otimes 1: Ff \otimes 1_{Gb} \Rightarrow F'f \otimes 1_{Gb}, \\
(\alpha \otimes \beta)_{1 \otimes g} & = & 1 \otimes \beta_{g}:1_{Fa} \otimes Gg \Rightarrow 1_{Fa} \otimes G'g.
\end{array}
\]
In the case that we consider the 1-cell $1_{a} \otimes 1_{b}$, both of these definitions agree to give that
\[
(\alpha \otimes \beta)_{1 \otimes 1} = 1_{Fa} \otimes 1_{Gb} = 1_{(Fa,Gb)}
\]
by the icon axioms since $F,G$ are both 2-functors; this also verifies the icon unit axiom for $\alpha \otimes \beta$.  To extend this over composition, we must check that it respects the two relations on 1-cells in $A \otimes B$:
\[
\begin{array}{rcl}
(f \otimes 1)*(f' \otimes 1)&  = & (f*f') \otimes 1, \\
(1 \otimes g) * (1 \otimes g') & = & 1 \otimes (g * g').
\end{array}
\]
This is easy to compute using the icon axioms, the axioms for the Gray tensor product, and the fact that $F,G$ are strict 2-functors, as shown below for the first relation.
\[
\begin{array}{rcl}
(\alpha \otimes \beta)_{f \otimes 1} * (\alpha \otimes \beta)_{f' \otimes 1} & = & (\alpha_{f} \otimes 1)*(\alpha_{f'} \otimes 1) \\
& = & (\alpha_{f} * \alpha_{f'}) \otimes 1 \\
& = & \alpha_{f*f'} \otimes 1 \\
& = & (\alpha \otimes \beta)_{(f*f') \otimes 1}
\end{array}
\]
Thus we can define the icon $\alpha \otimes \beta$ as above on generating 1-cells, and extend it to an icon by setting $(\alpha \otimes \beta)_{f*g} = (\alpha \otimes \beta)_{f} * (\alpha \otimes \beta)_{g}$ for composite 1-cells.  This completes the definition of the functor $\otimes$ on cells.

Next we must show that $\otimes$ is a 2-functor.  Now $(F \otimes G)*(F' \otimes G') = (F*F') \otimes (G*G')$ and $1_{A} \otimes 1_{B} = 1_{A \otimes B}$, so $\otimes$ strictly preserves composition and units at the level of 1-cells. The equality
\[
(\alpha_{f} \otimes 1)*(\alpha'_{f} \otimes 1) = (\alpha_{f}*\alpha'_{f}) \otimes 1
\]
also shows that $(\alpha \otimes \beta)*(\alpha' \otimes \beta')$ and $(\alpha * \alpha') \otimes (\beta * \beta')$ have the same component at $f \otimes 1$; a similar calculation shows that these icons also have the same component at $1 \otimes g$.  Since they have the same component for every generating 1-cell, this shows that
\[
(\alpha \otimes \beta)*(\alpha' \otimes \beta') = (\alpha * \alpha') \otimes (\beta * \beta').
\]
It is immediately clear that the Gray tensor product of two identity icons is the identity, proving that $\otimes$ is a 2-functor.

Finally, we must show that the associativity and unit isomorphisms are 2-natural.  Since the \textit{category} of 2-categories and 2-functors is a monoidal category under the Gray tensor product, we know that these constraints satisfy naturality for 1-cells; we must therefore check the 2-dimensional aspect.  Thus given icons
\[
\begin{array}{c}
\alpha:F \Rightarrow F':A \rightarrow A',\\
\beta:G \Rightarrow G':B \rightarrow B',\\
\gamma:H \Rightarrow H':C \rightarrow C',
\end{array}
\]
we must check that
\[
a_{A',B',C'} * \Big( (\alpha \otimes \beta) \otimes \gamma \Big) = \Big( \alpha \otimes (\beta \otimes \gamma) \Big) * a_{A,B,C}
\]
as icons.  Therefore we must check that they have the same components on 1-cells, which by the icon axioms reduces to checking that they have the same components on generating 1-cells.  For a generating 1-cell of the form $(f \otimes 1) \otimes 1$, the component of $(\alpha \otimes \beta) \otimes \gamma$ is $(\alpha_{f} \otimes 1) \otimes 1$.  Horizontal composition with the 2-functor $a_{A',B',C'}$ then gives the component
\[
a_{A',B',C'} \Big( (\alpha_{f} \otimes 1) \otimes 1 \Big) = \alpha_{f} \otimes (1 \otimes 1).
\]
On the other hand, the component of $\Big( \alpha \otimes (\beta \otimes \gamma) \Big) * a_{A,B,C}$ at $(f \otimes 1) \otimes 1$ is just the component of $\Big( \alpha \otimes (\beta \otimes \gamma) \Big)$ at
\[
a_{A,B,C}\Big( (f \otimes 1) \otimes 1 \Big) = f \otimes (1 \otimes 1),
\]
and these are clearly equal.  Analogous computation for the two other kinds of generating 1-cells then show that the associator is 2-natural.  The unit isomorphisms can be shown to be 2-natural in a similar fashion.
\end{proof}

\begin{remark}
It is relatively simple to extend the strongly monoidal structures in Corollary \ref{Iconsmon} and Proposition \ref{Graysmon} to symmetric ones by noting that the symmetry isomorphisms are 2-natural.  Thus we get a \textbf{Cat}-enriched symmetric structure, which can then be viewed as a (special kind of) symmetric monoidal structure on the underlying bicategories.
\end{remark}

\section{Strictification}

This section will review the construction of the strictification functor $\textrm{st}: \mb{Bicat} \rightarrow \mb{2Cat}$.  Doing so produces the functor between \textit{categories}, but that will not suffice for our purposes later.  Thus we will also extend $\textrm{st}$ to a 2-functor $\textrm{st}:\mb{Icon} \rightarrow \mb{Gray}_{icon}$, and in fact show that this functor is left adjoint to the inclusion 2-functor $i: \mb{Gray}_{icon} \rightarrow \mb{Icon}$.  In order to construct the unit and counit for this adjunction, we will use the explicit construction of biequivalences $e:\textrm{st}B \rightarrow B, f:B \rightarrow \textrm{st}B$ for any bicategory $B$.  The main technical tool in this section is the coherence theorem for bicategories, and we refer the reader to \cite{js, mp} for a review.

It should be noted that the main result of this section, Theorem \ref{stleft2adj}, is not quite a consequence of the general theory of 2-monads as developed in \cite{bkp, la2}.  The general theory gives at least two ways to show that the inclusion $\mb{Gray}_{icon} \hookrightarrow \mb{Icon}$ has a left 2-adjoint.  One method shows that each of the two inclusions in the composite
\[
\mb{Gray}_{icon} \hookrightarrow \mb{Icon}_{s} \hookrightarrow \mb{Icon}
\]
has a left 2-adjoint, the first because it is of the form $f:^{*}:S$-$\mb{Alg}_{s} \rightarrow T$-$\mb{Alg}_{s}$ for a map of 2-monads $f:T \rightarrow S$ and the second because it is of the form $T$-$\mb{Alg}_{s} \rightarrow T$-$\mb{Alg}$.  The other method to abstractly produce a left 2-adjoint uses that $\mb{Icon}$ is nearly the 2-category of pseudo-algebras for the 2-monad $S$ whose strict algebras are 2-categories.  Unfortunately, neither of these methods immediately shows that the left adjoint is the familiar strictification functor st.  Furthermore, the explicit form of the unit and counit of this adjunction will be useful later, so we verify this adjunction using elementary means.

Let $B$ be a bicategory.  We will begin by constructing a 2-category $\textrm{st}B$, and then show that it is biequivalent to the original bicategory $B$.  In order to define the 2-cells of $\textrm{st}B$, it will be necessary to define the action of the functor $e:\textrm{st}B \rightarrow B$ on the 0- and 1-cells.  Now for the definition.  The 2-category $\textrm{st}B$ will have the same objects as $B$.  A 1-cell from $a$ to $b$ will be a string of composable 1-cells of $B$ starting at $a$ and ending at $b$.  There is unique empty or length 0 string $\varnothing_{a}$ from $a$ to $a$ for each object, and this will serve as the identity 1-cell.

Now we will define the function on underlying 0- and 1-cells of $e$.  On 0-cells, $e$ is the identity function.  On 1-cells, we define
\[
e(f_{n} f_{n-1} f_{1}) = (\cdots (f_{n}*f_{n-1})*f_{n-2}) \cdots *f_{2})*f_{1};
\]
for the empty string $\varnothing:a \rightarrow a$, we set $e(\varnothing) = I_{a}$.
The set of 2-cells between the strings $f_{n}f_{n-1}\cdots f_{1}$ and $g_{m}g_{m-1}\cdots g_{1}$ is defined to be the set of 2-cells between $e(f_{n}f_{n-1} \cdots f_{1})$ and $e( g_{m}g_{m-1}\cdots g_{1})$ in $B$.  It is now obvious how $e$ acts on 2-cells.

We will now give $\textrm{st}B$ the structure of a 2-category.  Composition of 1-cells is given by concatenation of strings, with the empty string as the identity.  It is immediate that this is strictly associative and unital.  Vertical composition of 2-cells is as in $B$, and this is strictly associative and unital since vertical composition of 2-cells in a bicategory is always strict in this way.

For horizontal composition of 2-cells, note that coherence for bicategories implies that there is a unique coherence isomorphism
\[
e(f_{n}\cdots f_{1})*e(g_{m}\cdots g_{1}) \cong e(f_{n}\cdots f_{1} g_{m} \cdots g_{1}).
\]
Thus we can now define the horizontal composition $\alpha * \beta$ in $\textrm{st}B$ as the composite
\[
\begin{array}{rcl}
e(f_{n}\cdots f_{1}g_{m}\cdots g_{1})  & \cong  & e(f_{n}\cdots f_{1})e(g_{m}\cdots g_{1}) \\
& \stackrel{\alpha * \beta}{\longrightarrow}  & e(f'_{n}\cdots f'_{1})e(g'_{m}\cdots g'_{1}) \\
& \cong  & e(f'_{n}\cdots f'_{1}g'_{m}\cdots g'_{1})
\end{array}
\]
in $B$, where the unlabeled isomorphisms are the unique coherence isomorphisms mentioned above.  The uniqueness of these isomorphisms ensures that this definition satisfies the middle-four interchange laws as well as being strictly associative and unital.  This completes the proof that $\textrm{st}B$ is a 2-category.

Now we will finish proving that $e: \textrm{st}B \rightarrow B$ is a functor, and in particular a bijective-on-objects biequivalence.  By construction, $e$ is functorial on vertical composition of 2-cells.  The unit isomorphism is an invertible 2-cell $I_{a} \Rightarrow  e(\varnothing_{a})$, so we define it to be the identity.  The constraint cell for composition is an isomorphism
\[
e(f_{n}\cdots f_{1})*e(g_{m} \cdots g_{1}) \cong e(f_{n}\cdots f_{1}g_{m}\cdots g_{1})
\]
which once again is the unique such isomorphism given by coherence for bicategories.  Coherence also implies that the functor axioms hold.  To show that $e$ is a bijective-on-objects biequivalence, we only have to show that it is locally an equivalence of categories.  But by definition, it is locally full and faithful since $B$ and $\textrm{st}B$ have the same 2-cells.  This functor is also locally essentially surjective, as every 1-cell in $B$ is the image under $e$ of a length 1 string.

We will also need the biequivalence $f:B \rightarrow \textrm{st}B$ in order to construct the 2-adjunction between $\mb{Icon}$ and $\mb{Gray}_{icon}$.  The functor $f$ is defined to be the identity on objects, to include every 1-cell as a string of length one, and to be the identity function on 2-cells.  It is then simple to show that $f$ is a bijective-on-objects biequivalence.  It should be noted that $ef = 1_{B}$, and $fe$ is biequivalent to $1_{\textrm{st}B}$ in $\mathbf{Bicat}(\textrm{st}B, \textrm{st}B)$ by a transformation whose components on objects can all be taken to be identities and whose components on 1-cells all come from coherence; this in fact shows that $fe$ is isomorphic to $1_{\textrm{st}B}$ in $\mathbf{Icon}$, proving that every bicategory is \textit{equivalent} to a 2-category in $\mathbf{Icon}$.  This result was first noted by Lack and Paoli in \cite{lp}.

We will now use the coherence theorem for functors to construct a 2-functor $\textrm{st}F: \textrm{st}B \rightarrow \textrm{st}B'$ given a functor $F: B \rightarrow B'$ between bicategories.  We define $\textrm{st}F : \textrm{st}B \rightarrow \textrm{st}B'$ as follows.  On 0-cells, $\textrm{st}F$ agrees with $F$.  On 1-cells, we define
\[
\textrm{st}F(f_{n} \cdots f_{1}) = Ff_{n} \cdots Ff_{1},
\]
and $\textrm{st}F(\varnothing_{a}) = \varnothing_{Fa}$.  Now let $\alpha:e(f_{n} \cdots f_{1}) \Rightarrow e(g_{m} \cdots g_{1})$ be a 2-cell in $\textrm{st}B$.  Then we define $\textrm{st}F(\alpha)$ to be the 2-cell
\[
e(Ff_{n} \cdots Ff_{1}) \cong F\Big( e(f_{n} \cdots f_{1}) \Big) \stackrel{F\alpha}{\longrightarrow} F\Big( e(g_{m} \cdots g_{1}) \Big) \cong e(Fg_{m} \cdots Fg_{1}),
\]
where the unlabeled isomorphisms are the unique isomorphism 2-cells provided by coherence for functors.

One then uses coherence for functors to show that this is a strict functor, and that $\textrm{st}(F \circ G) = \textrm{st}F \circ \textrm{st}G$. The commutativity of the square
\[
\xy
{\ar^{F} (0,0)*+{X}; (20,0)*+{Y} };
{\ar_{f} (0,0)*+{X}; (0,-12)*+{\textrm{st} X} };
{\ar^{f} (20,0)*+{Y}; (20,-12)*+{\textrm{st} Y} };
{\ar_{\textrm{st} F} (0,-12)*+{\textrm{st} X}; (20,-12)*+{\textrm{st} Y} }
\endxy
\]
is immediate from the definitions.  It is not the case that $F \circ e = e \circ \textrm{st}F$, but there is an invertible icon $\omega$ between these with each component given by the unique coherence 2-cell; in particular, this equation does hold when $X,Y$ are 2-categories and $F$ is a 2-functor, as the unique coherence 2-cell is necessarily an identity.

\begin{thm}\label{stleft2adj}
The assignment $B \mapsto \textrm{st}B$ can be extended to a 2-functor
\[
\textrm{st}:\mathbf{Icon} \rightarrow \mathbf{Gray}_{icon}.
\]
This 2-functor is the left 2-adjoint to the inclusion $\mathbf{Gray}_{icon} \hookrightarrow \mathbf{Icon}$.
\end{thm}
\begin{proof}
To give $B \mapsto \textrm{st}B$ the structure of a 2-functor, we must define it on higher cells.  We have given the construction of $\textrm{st}F$ above, so let $\alpha:F \rightarrow G$ be an icon between functors.  Since $\textrm{st}F(b) = F(b)$, we have that $\textrm{st}F$ and $\textrm{st}G$ agree on objects, so we can define icons between them.  Let $f_{1} \cdots f_{n} $ be a 1-cell in $\textrm{st}B$ from $a$ to $b$.  Then
\[
\textrm{st}F\Big( f_{1} \cdots f_{n}  \Big) = Ff_{1} Ff_{2} \cdots Ff_{n}
\]
by definition, so we define $\textrm{st}\alpha_{ f_{1} \cdots f_{n}  }$ to be the 2-cell represented by $\alpha_{f_{1}}* \cdots *\alpha_{f_{n}}$.  For the empty string from $a$ to $a$, define $\textrm{st}\alpha_{\varnothing} = 1_{I_{Fa}}$.  It is immediate that this is an icon $\textrm{st}\alpha: \textrm{st}F \Rightarrow \textrm{st}G$.

We have already noted that $\textrm{st}(FG) = \textrm{st}F \textrm{st}G$, and it is clear that $\textrm{st}(1_{B}) = 1_{\textrm{st}B}$.  We only must check that $\textrm{st}$ preserves horizontal composition of icons, or that
\[
\textrm{st}\alpha * \textrm{st}\beta = \textrm{st}(\alpha * \beta).
\]
Writing this down in terms of components, it becomes clear that this follows immediately from the definitions and interchange.  Thus $\textrm{st}:\mathbf{Icon} \rightarrow \mathbf{Gray}_{icon}$ is a 2-functor.

To prove that $\textrm{st}$ is the left 2-adjoint to the incluion $\mathbf{Gray}_{icon} \hookrightarrow \mathbf{Icon}$, we must give a unit and counit for the adjunction.  Denoting the inclusion by $i$, the counit would be a 2-natural transformation with components $\textrm{st} \, i (A) \rightarrow A$ for a 2-category $A$ and the unit would be a 2-natural transformation with components $B \rightarrow i \, \textrm{st} (B)$ for a bicategory $B$.  The counit is defined to be $e:\textrm{st}A \rightarrow A$, which is a 2-functor when $A$ is a 2-category, and the unit is defined to be $f:B \rightarrow \textrm{st}B$.  The 1-dimensional aspect of naturality for these components is contained in the remarks above, so we only need to check the 2-dimensional aspect.  For $f$, the component of $\textrm{st}\alpha * f$ at a 1-cell $r:x \rightarrow y$ is the 2-cell $\{ Fr \} \Rightarrow \{ Gr \}$ represented by $\alpha_{r}$, and it is trivial to show that the same 2-cell is the component of $f* \alpha$; the calculation for $e$ is analogous, confirming 2-naturality.

Finally, we must check the triangle identities.  One reduces to showing that
\[
A \stackrel{f}{\longrightarrow} \textrm{st} A \stackrel{e}{\longrightarrow} A
\]
is the identity for any 2-category $A$, which we have already noted holds.  The other involves showing that
\[
\textrm{st}B \stackrel{\textrm{st}f}{\longrightarrow} \textrm{st} \textrm{st} B \stackrel{e_{\textrm{st}B}}{\longrightarrow} \textrm{st}B
\]
is the identity for any bicategory $B$.  This is simple to check on cells, so it only remains to show that the constraint isomorphisms are identities, but this is straightforward to check from the definition.
\end{proof}

\section{Monoidal structure}

In this final section, we will prove that the strictification functor $\textrm{st}:\mb{Icon} \rightarrow \mb{Gray}_{icon}$ is monoidal at the level of bicategories.  It is for the proof of the monoidal structure that we require the use of icons, as indicated by Proposition \ref{stpsnat} below.  We begin by first showing that the functor st is \textit{not} monoidal in the 1-dimensional sense.  The following proposition originally appeared in \cite{gps}, although its statement there did not use the language of icons.  As a proof can be found in any of \cite{gps, gthesis, gbook}, we omit it here.

\begin{proposition} \label{bicatcompprop}
Let $X,Y$ be bicategories.  Then there exists a cubical functor
$\hat{\textrm{st}}: \textrm{st}X \times \textrm{st}Y \rightarrow \textrm{st}(X
\times Y)$
such that
\begin{enumerate}
\item $\hat{\textrm{st}}$ is the identity on objects and
\item there is an invertible icon $\zeta$ as pictured below.
\[
\xy
{\ar^{\hat{\textrm{st}}} (0,0)*+{\textrm{st}X \times \textrm{st}Y}; (15,15)*+{\textrm{st}(X \times Y)} };
{\ar_{e_{X} \times e_{Y}} (0,0)*+{\textrm{st}X \times \textrm{st}Y}; (30,0)*+{X \times Y} };
{\ar^{e_{X \times Y}} (15,15)*+{\textrm{st}(X \times Y)}; (30,0)*+{X \times Y} };
{\ar@{=>}^{\zeta} (15,8)*+{}; (15,2)*+{} }
\endxy
\]
\end{enumerate}
\end{proposition}

For the next proposition, we will need to know explicitly what $\hat{\textrm{st}}$ does to 1-cells.  If we let $\{ f_{n}, \ldots, f_{1} \}$ be a 1-cell in $\textrm{st}X$ and let $\{ g_{m}, \ldots, g_{1} \}$ be a 1-cell in $\textrm{st}Y$, then the required formula is
\[
\hat{\textrm{st}} \left( \{ f_{n}, \ldots, f_{1} \}, \{ g_{m}, \ldots, g_{1} \} \right) = (1,g_{m})(1,g_{m-1}) \cdots (1, g_{1})(f_{n},1) \cdots (f_{1}, 1).
\]

\begin{proposition}\label{stpsnat}
The cubical functor $\hat{\textrm{st}}$
is natural in both variables, up to an invertible icon.
\end{proposition}
\begin{proof}
Let $F:X \rightarrow X', G:Y \rightarrow Y'$ be functors between bicategories.  To prove the naturality of $\hat{\textrm{st}}$, we must construct an invertible icon in the square below.
\[
\xy
{\ar^{\hat{\textrm{st}}} (0,0)*+{\textrm{st}X \times \textrm{st}Y}; (40,0)*+{\textrm{st}(X \times Y)} };
{\ar^{\textrm{st}(F \times G)} (40,0)*+{\textrm{st}(X \times Y)}; (40,-25)*+{\textrm{st}(X' \times Y')} };
{\ar_{\textrm{st}F \times \textrm{st}G} (0,0)*+{\textrm{st}X \times \textrm{st}Y}; (0,-25)*+{\textrm{st}X' \times \textrm{st}Y'} };
{\ar_{\hat{\textrm{st}}} (0,-25)*+{\textrm{st}X' \times \textrm{st}Y'}; (40,-25)*+{\textrm{st}(X' \times Y')} };
(20,-12.5)*{\cong}
\endxy
\]

First, it is clear that these agree on objects since both functors $\hat{\textrm{st}}$ are the identity on objects and $\textrm{st}F$ agrees with $F$ on objects, so it is possible to define an icon between them.

By definition, $\hat{\textrm{st}}(f,g) = \hat{\textrm{st}}(\varnothing,g)*\hat{\textrm{st}}(f,\varnothing)$ for a 1-cell $(f,g)$ in $\textrm{st}X \times \textrm{st}Y$.  Thus the top composite cubical functor gives
\[
(F1, Gg_{m})(F1, Gg_{m-1})\cdots (F1, Gg_{1})(Ff_{n}, G1) \cdots (Ff_{1}, G1)
\]
while the composite along the left and bottom is
\[
(1, Gg_{m})(1, Gg_{m-1})\cdots (1, Gg_{1})(Ff_{n}, 1) \cdots (Ff_{1}, 1).
\]
We thus define the desired invertible icon as a horizontal composite of unit constraints for $F$ and $G$.  Naturality follows from the naturality of the associators in the bicategories $X', Y'$ together with the interchange law.  We leave the icon axioms to the reader, as they are entirely straightforward.
\end{proof}

\begin{remark}
The proof of the previous proposition shows that $\hat{\textrm{st}}$ is \textit{not} natural in the sense of giving a natural transformation in the standard, 1-categorical sense, but only up to an icon.  In particular, this shows that st is not a lax monoidal functor between monoidal \textit{categories}.
\end{remark}

In order to avoid complicated computations, we will employ the theory of doctrinal adjunction as in \cite{k}.  We begin by recalling the basic result.  Let $K$ be a 2-category, and $T$ a 2-monad on $K$.  Assume we have $T$-algebras $x,y$, and assume we have an adjunction $f \dashv u$ in $K$ with $f:x \rightarrow y, u: y \rightarrow x$.  Then there is a bijection, given by taking mates, between data making $u$ into a lax $T$-algebra morphism and data making $f$ into a oplax $T$-algebra morphism.  In our case, the relevant 2-category is $\mb{2CAT}$, the 2-category of (potentially large) 2-categories, 2-functors, and 2-natural transformations, and the 2-monad $T$ is the free strongly monoidal 2-category 2-monad.  We have algebras $\mb{Icon}$ and $\mb{Gray}_{icon}$, and a 2-adjunction between them.  We begin employing the theory of doctrinal adjunction by giving a lax structure on $i$ in the following proposition.

\begin{proposition}
The 2-functor $i: \mb{Gray}_{icon} \hookrightarrow \mb{Icon}$ is $\mb{Cat}$-lax monoidal.
\end{proposition}
\begin{proof}
Since $i$ sends the unit object in $\mb{Gray}_{icon}$ to the unit object in $\mb{Icon}$, we must give a 2-natural transformation with components $iA \times iB \rightarrow i(A \otimes B)$ and then check the lax functor axioms.  Now $A,B$ are 2-categories, but the components above are in $\mb{Icon}$, hence are pseudofunctors; thus we define these to be the universal cubical functors $c:A \times B \rightarrow A \otimes B$.  The proof that these components give a 2-natural transformation is a consequence of Lemma \ref{cubicalicons}.  There are now two unit axioms and one associativity axiom to check.

The left unit axiom is the commutativity of the diagram below, which is trivial to check.
\[
\xy
{\ar^{c} (0,0)*+{1 \times A}; (30,0)*+{1 \otimes A} };
{\ar^{\cong} (30,0)*+{1 \otimes A}; (30,-15)*+{A} };
{\ar_{\cong} (0,0)*+{1 \times A}; (30,-15)*+{A} };
\endxy
\]
The right unit axiom is analogous.  The associativity axiom is the commutativity of the following diagram.
\[
\xy
{\ar^{c \times 1} (0,0)*+{(A \times B) \times C}; (40,0)*+{(A \otimes B) \times C} };
{\ar^{c} (40,0)*+{(A \otimes B) \times C}; (80,0)*+{(A \otimes B) \otimes C} };
{\ar^{\cong} (80,0)*+{(A \otimes B) \otimes C}; (80,-20)*+{A \otimes (B \otimes C)} };
{\ar_{\cong} (0,0)*+{(A \times B) \times C}; (0,-20)*+{A \times (B \times C)} };
{\ar_{1 \times c} (0,-20)*+{A \times (B \times C)}; (40,-20)*+{A \times (B \otimes C)} };
{\ar_{c} (40,-20)*+{A \times (B \otimes C)}; (80,-20)*+{A \otimes (B \otimes C)} };
\endxy
\]
This diagram clearly commutes on objects, and it is easy to check that it commutes on 1- and 2-cells which have only a single non-identity component.  Since both composites are cubical, this implies that they agree on all 1- and 2-cells.  The constraints for both are given as composites of the canonical isomorphisms
\[
(f \otimes 1) \circ (1 \otimes g) \cong (1 \otimes g) \circ (f \otimes 1)
\]
or as identities, and once again it is straightforward to verify that the Gray tensor product axioms force these to be equal.
\end{proof}

Thus we have the following immediate consequence.

\begin{corollary}\label{oplaxst}
The 2-functor $\textrm{st}: \mb{Icon} \rightarrow \mb{Gray}_{icon}$ is $\mb{Cat}$-oplax monoidal.
\end{corollary}

We need one final preliminary result before proving the main theorem.

\begin{proposition}\label{chiequiv}
The strict 2-functor $\chi_{XY}:\textrm{st}X \otimes \textrm{st}Y \rightarrow \textrm{st}(X
\times Y)$ induced by $\hat{\textrm{st}}$ is an internal equivalence in $\mathbf{Gray}_{icon}$.
\end{proposition}
\begin{proof}
We must produce a 2-functor $\overline{\chi}: \textrm{st}(X \times Y) \rightarrow \textrm{st}X \otimes \textrm{st}Y$ and invertible icons $\overline{\chi} \chi \cong 1, \chi \overline{\chi} \cong 1$.  Now by Proposition \ref{bicatcompprop}, we get that $\hat{\textrm{st}}$ is a bijective-on-objects biequivalence because both $e_{X \times Y}$ and $e_{X} \times e_{Y}$ are, and bijective-on-objects biequivalence satisfy the 2-out-of-3 property in the following sense:  if $fg \cong h$ in $\mathbf{Icon}$, then all three functors are bijective-on-objects biequivalences if two of them are.  This shows, by Lemma \ref{cubicalicons}, that $\chi$ is also a bijective-on-objects biequivalence.

To construct the 2-functor $\overline{\chi}$, we will use Theorem \ref{stleft2adj}.  The 2-functor
\[
\overline{\chi}: \textrm{st}(X \times Y) \rightarrow \textrm{st}X \otimes \textrm{st}Y
\]
determines, and is determined by, a functor $j:X \times Y \rightarrow \textrm{st}X \otimes \textrm{st}Y$ such that $\overline{\chi} f_{X \times Y} = j$.  Define the functor $j$ to be the composite below.
\[
X \times Y \stackrel{f_{X} \times f_{Y}}{\longrightarrow} \textrm{st}X \times \textrm{st}Y \stackrel{c}{\longrightarrow} \textrm{st}X \otimes \textrm{st}Y
\]
By Lemma \ref{universalboobieq}, $c$ is a bijective-on-objects biequivalence, and $f_{X} \times f_{Y}$ is as well.  Once again by the 2-out-of-3 property, $\overline{\chi}$ is then a bijective-on-objects biequivalence.

Now bijective-on-objects biequivalences are the internal equivalences in $\mathbf{Icon}$, and a 2-functor in $\mathbf{Gray}_{icon}$ is an internal equivalence if and only if it is an internal equivalence in $\mathbf{Icon}$ and there is a pseudoinverse which is a 2-functor.  This implies that both $\chi$ and $\overline{\chi}$ are internal equivalences in $\mathbf{Icon}$.  If we show that $\overline{\chi} \chi \cong 1$, then by the uniqueness (up to isomorphism) of pseudoinverses we will get that $\chi \overline{\chi} \cong 1$ as well.  Once again by Lemma \ref{cubicalicons}, to show that $\overline{\chi} \chi \cong 1$, we will prove that $\overline{\chi} \chi c \cong c$.  By the definition of $\overline{\chi}$, we have $\overline{\chi} f_{X \times Y} = c (f_{X} \times f_{Y})$, so
\[
\overline{\chi} \cong \overline{\chi} f_{X \times Y} e_{X \times Y} = c (f_{X} \times f_{Y})e_{X \times Y}.
\]
Thus we have the following calculation which finishes the proof.
\[
\begin{array}{rcl}
\overline{\chi}\chi c & \cong & c (f_{X} \times f_{Y}) e_{X \times Y} \chi c \\
& = & c (f_{X} \times f_{Y}) e_{X \times Y} \hat{\textrm{st}} \\
& \stackrel{11\zeta}{\cong} & c (f_{X} \times f_{Y}) (e_{X} \times e_{Y}) \\
& \cong & c
\end{array}
\]
\end{proof}

\begin{thm}\label{stmonoidal}
The 2-functor $\textrm{st}:\mathbf{Icon} \rightarrow \mathbf{Gray}_{icon}$ can be given the structure of a  symmetric monoidal functor between the underlying monoidal bicategories.
\end{thm}
\begin{proof}
By Corollary \ref{oplaxst}, we already know that the functor $\textrm{st}:\mb{Icon} \rightarrow \mb{Gray}_{icon}$ is $\mb{Cat}$-oplax monoidal.  The structure maps for the $\mb{Cat}$-lax monoidal structure on $i$ are the universal cubical functors $c_{AB}$, and we know these are internal equivalences in $\mb{Icon}$ by Lemma \ref{universalboobieq}.  Unfortunately, the unit and counit for the adjunction $\textrm{st} \dashv i$ are not both internal equivalences:  the counit is the 2-functor $e:\textrm{st} \, i A \rightarrow A$ which is not an internal equivalence in $\mb{Gray}_{icon}$ since it has no pseudo-inverse which is also a 2-functor.  But the formula given in Lemma \ref{oplaxtomonoidal1} for the oplax structure constraints is easily checked to be the 2-functor $\overline{\chi}$ in the proof of Proposition \ref{chiequiv}.  That theorem shows that the 2-functor $\textrm{st}$ satisfies the hypotheses of Lemma \ref{oplaxtomonoidal2}, and moreover that we can take the structure map $\textrm{st}A \otimes \textrm{st}B \rightarrow \textrm{st}(A \times B)$ to be the 2-functor $\chi$.

To give this monoidal functor the structure of a braided monoidal functor, we must provide an invertible modification $U$.  By the isomorphism in Lemma \ref{cubicalicons}, such an invertible modification will correspond to one with components as below.
\[
\xy
{\ar^{\cong} (0,0)*+{\textrm{st}X \times \textrm{st}Y}; (50,0)*+{\textrm{st}Y \times \textrm{st}X} };
{\ar^{\hat{\textrm{st}}} (50,0)*+{\textrm{st}Y \times \textrm{st}X}; (50,-20)*+{\textrm{st}(Y \times X)} };
{\ar_{\hat{\textrm{st}}} (0,0)*+{\textrm{st}X \times \textrm{st}Y}; (0,-20)*+{\textrm{st}(X \times Y)} };
{\ar_{\textrm{st}(\cong)} (0,-20)*+{\textrm{st}(X \times Y)}; (50,-20)*+{\textrm{st}(Y \times X)} };
(25,-10)*{\Downarrow}
\endxy
\]
Let $(f,g)$ be a 1-cell in $\textrm{st}X \times \textrm{st}Y$.  Then the image of this 1-cell using the top and right functors is the composite string $(1, f)(g,1)$, while the image using the left and bottom functors is the composite string $(g,1)(1, f)$.  To give an invertible 2-cell between these, we must give an invertible 2-cell between their images under $e$.  Examining $e \big( (1, f)(g,1) \big)$ and $e \big( (g,1)(1, f) \big)$, we see that there is a unique coherence isomorphism (arising from the bicategory $X \times Y$) between them, so we define that to be the component of $U$ at $(f,g)$.

There are now two axioms to check for the braided monoidal structure of this functor (see \cite{gur2} for the axioms in the fully weak case).  In both cases, the same argument shows that the axiom holds.  First, all of the 2-cells in each pasting are the identity, except for (possibly) the naturality isomorphisms for $\chi$ and the components of $U$.  For the naturality isomorphisms for $\chi$, the 1-cells involved all have identity unit constraints, hence by the definition these naturality isomorphisms are also the identity.  Now the components of $U$ are not identities, but they all evaluate to unique coherence isomorphisms under $e_{X \times Y}:\textrm{st}(X \times Y) \rightarrow X \times Y$.  We know that $e$ is a biequivalence, and in particular a pair of 2-cells in $\textrm{st}X$ are equal if and only if they are parallel and have the same image under $e_{X}$.  Thus by the above, we see that applying $e$ to the pastings of the braided monoidal functor axioms will produce a pair of parallel 2-cells which are both composites of coherence isomorphisms and therefore equal by the coherence theorem for bicategories.  Thus the pastings in question will be equal after applying $e$, hence equal beforehand as well.  This shows that $\textrm{st}$ is a braided monoidal functor.

Finally, we complete the proof that $\textrm{st}$ is symmetric monoidal.  This involves checking a single axiom (see \cite{ds} or \cite{sp}), showing that $\textrm{st}$ preserves the symmetry.  As with the braiding, it is easy to check that all of the cells involved in the pasting are identities or are represented by unique coherence cells.  Thus coherence for bicategories ensure that this axiom holds, and so $\textrm{st}$ is symmetric monoidal.
\end{proof}

\begin{remark}
At this point, one could go on to show that $\textrm{st}:\mb{Bicat} \rightarrow \mb{Gray}$ is a symmetric monoidal functor between symmetric monoidal tricategories.  The underlying functor of tricategories is discussed in \cite{gthesis, gbook}, and we have established the symmetric monoidal structure above.  All that remains is to interpret the above proof in the context of monoidal tricategories, and then to insert coherence cells to take care of the proliferation of units that would occur when changing the invertible icons used here to pseudonatural transformations with identity components.
\end{remark}

\end{document}